\documentclass[11pt,reqno, twoside]{article}
\usepackage[letterpaper,margin=1.3in,footskip=0.70in]{geometry}
\usepackage{microtype}
\usepackage{amsmath,amssymb,amsthm}
\usepackage{mathtools}
\usepackage[dvipsnames]{xcolor}
\usepackage{enumitem}
\setlist{noitemsep}
\usepackage{mathrsfs}
\usepackage{tikz-cd}
\usepackage{array}
\usepackage[pdfusetitle,colorlinks]{hyperref}
\hypersetup{citecolor=black,linkcolor=black,urlcolor=black}
\usepackage[capitalise,noabbrev]{cleveref}
\crefformat{equation}{\ensuremath{(#2#1#3)}}
\crefmultiformat{equation}{\ensuremath{(#2#1#3)}}{ and~\ensuremath{(#2#1#3)}}{, \ensuremath{(#2#1#3)}}{, and~\ensuremath{(#2#1#3)}}
\numberwithin{figure}{subsection}
\numberwithin{equation}{subsection}
\newtheorem{theorem}[figure]{Theorem}
\newtheorem{lemma}[figure]{Lemma}
\newtheorem{corollary}[figure]{Corollary}

\theoremstyle{definition}

\theoremstyle{definition}
\newtheorem{remark}[figure]{Remark}
\theoremstyle{definition}

\theoremstyle{definition}

\theoremstyle{cited}

\usepackage{lipsum}
\usepackage{fancyhdr}
\fancypagestyle{mypagestyle}{
\fancyhead[OC]{\footnotesize SHUBHODIP MONDAL}
\fancyhead[EC]{\footnotesize  RECONSTRUCTION OF THE STACKY APPROACH TO DE RHAM COHOMOLOGY}
\fancyfoot[C]{\thepage}

\vspace{-10mm}
}

\newcommand{\w}{\text}

\newcolumntype{C}[1]{>{\centering\arraybackslash}p{#1}}

\title{\large \textbf{RECONSTRUCTION OF THE STACKY APPROACH TO DE RHAM COHOMOLOGY}}
\author{\small SHUBHODIP MONDAL}
\date{}
\begin{document}

\maketitle
\begin{abstract}
In this short paper, we use Tannakian reconstruction techniques to prove a result that explains how to reconstruct the stacky approach to de Rham cohomology from the classical theory algebraic de Rham cohomology via an application of the adjoint functor theorem. 
\end{abstract}

\let\thefootnote\relax\footnotetext{\\Affiliation: University of Michigan, Ann Arbor, MI, USA. \\
Email: smondal@umich.edu\\
\textbf{MSC classes}: 14F30, 14F40, 20G05 \\}

\section{\small{INTRODUCTION}} Let $k$ be a perfect field of characteristic $p >0.$ The stacky approach to de Rham cohomology due to Drinfeld recovers the theory of de Rham cohomology via cohomology of the structure sheaf of certain stacks \cite{drinfeld2018stacky}. More precisely, for a smooth scheme $X$ over $k,$ he constructs a stack $X^{\mathrm{dR}}$ such that $R\Gamma_{\mathrm{dR}}(X) \simeq R\Gamma (X^{\mathrm{dR}}, \mathscr O).$ The case of $X = \mathbb{A}_k^1$ here is particularly important. The ring scheme structure on $\mathbb{A}^{1}_k$ induces a ring stack structure on $\mathbb A^{1,\mathrm{dR}}$ which can be used to determine the stacks $X^{\mathrm{dR}}$ for any smooth scheme $X$ over $k$ (see \cite[Section 2.4]{LM21}). Let us use $\mathbb{G}_a$ to denote $\mathbb{A}^1_k$ with the enhancement of a ring scheme and $\mathbb{G}_a^{\w{dR}}$ to denote $\mathbb{A}^{1, \w{dR}}$ equipped with the enhancement of a ring stack. The ring stack $\mathbb{G}_a^{\mathrm{dR}}$ has been used in \cite[Theorem 1.6]{LM21} to give a proof of Drinfeld's refinement of the Deligne--Illusie theorem~(\textit{c.f.}~\cite[Remark 5.16]{BhaLur2}). Actually, all the endomorphisms of de Rham cohomology as a functor has been classified in \cite[Theorem 4.24]{LM21}, which really uses the ring stack $\mathbb{G}_a^{\mathrm{dR}}.$ This, of course, also recovers the Sen operators, which was also only recently observed due to the work of Drinfeld and Bhatt--Lurie, which is further studied in \cite[Section 3.5]{BhaLur}. \vspace{2mm}

Given the fundamental nature of these new results, one might naturally arrive at a slightly philosophical question that Illusie asked the author: \textit{does the ring stack $\mathbb{G}_a^{\mathrm{dR}}$ give any information about the theory of de Rham cohomology that could not be seen otherwise?} The goal of this note is to prove that that is \textit{not} the case. The ring stack $\mathbb{G}_a^{\mathrm{dR}}$ is not an enrichment of the theory of de Rham cohomology, it is actually \textit{equivalent} to the theory of de Rham cohomology. In fact, it is “dual" to the theory of de Rham cohomology. \vspace{2mm}

The goal of this note is to prove the above claim precisely. To do so, we will develop the stacky approach to de Rham cohomology from the scratch by using the classical theory of algebraic de Rham cohomology. The tools we use to do so are only:

\begin{enumerate}

\item The classical algebraic de Rham complex, as considered in \cite{gro1}.

    \item One computation from derived de Rham cohomology, developed earlier by Illusie \cite{Ill2}, Beilinson \cite{bei} and Bhatt \cite{Bha12}.
    
    \item The theory of higher categories, as developed by Lurie in \cite{Lu}. In particular, we use the adjoint functor theorem and left Kan extensions.

\end{enumerate}{}
\textit{The main result is \cref{main} below.} We will need some preparations before stating it. Note that in above, $ \mathbb{G}_a^{\mathrm{dR}} := \mathrm{Cone}(\mathbb{G}_a^\sharp \to \mathbb{G}_a)$. For details on the cone construction, we refer to \cite[Section 1.3]{prismatization} or \cite[Section 2.1]{LM21}. The Tannakian reconstruction techniques used in our paper is useful in more general situations concerning derived categories of representations, see \cref{easthall}.

\subsection{Warning} This note is only about reconstruction of the stacky approach to \textit{de Rham cohomology.} We do not make any claims about reconstruction of the stacks $\Sigma, \Sigma', \Sigma''$ from \cite{prismatization} in the context of absolute prismatic cohomology.

\subsection{Notations and categorical prerequisites}

\begin{enumerate}\label{not}

\item We fix a prime $p.$ Let $k$ be a perfect field of characteristic $p.$ Let $\mathrm{Alg}_k$ denote the category of ordinary algebras and $\mathrm{Alg}_k^{\mathrm{sm}}$ denote the category of smooth $k$-algebras.

\item Let $\mathbb{G}_a$ denote the affine line $\mathbb{A}^1$ viewed with the enhancement of a ring scheme. Let $\mathbb{G}_a^\sharp$ denote the divided power completion of $\mathbb{G}_a$ at the origin, which has a group scheme structure. In fact, $\mathbb G_a^\sharp$ has the natural structure of a $\mathbb{G}_a$-module. See \cite[Section 3.2 - 3.5]{prismatization} for more details.

    \item We will let $\mathcal{S}$ denote the $\infty$-category of spaces, or equivalently the $\infty$-category of $\infty$-groupoids, or the $\infty$-category of anima.
    
    \item We let $\mathrm{ARings}_k$ denote the $\infty$-category of animated $k$-algebras. There is a forgetful functor $\mathrm{ARings}_k \to \mathcal{S}$ which preserves small limits.
    
    \item All schemes and stacks (in the classical sense) are identified with the functors they represent. More precisely, an object in the category $\mathrm{Fun}(\mathrm{Alg}_k, \mathcal{S})$ is simply called a \textit{stack}. An object in the category $\mathrm{Fun}(\mathrm{Alg}_k, \mathrm{ARings}_k)$ is simply called a \textit{ring stack.} There is a natural forgetful functor $$\mathrm{Fun}(\mathrm{Alg}_k, \mathrm{ARings}_k) \to \mathrm{Fun}(\mathrm{Alg}_k, \mathcal S).$$
    
    \item Let $C,D$ be two $\infty$-categories. Let $\mathrm{Fun}^L(C,D)$ be the category of functors that are left adjoints and let $\mathrm{Fun}^R(C,D)$ be the category of functors that are right adjoints. Then $\mathrm{Fun}^L(C,D)^{\mathrm{op}} \simeq {\mathrm{Fun}^R(C,D)}.$ \cite[Prop.~5.2.6.2]{Lu}.
    
    \item Let $C$ be a presentable $\infty$-category and let $F: \mathrm{ARings}_k \to C$ be a colimit preserving functor. By the adjoint functor theorem \cite[Prop.~5.5.2.9]{Lu}, we have a right adjoint $G: C \to \mathrm{ARings}_k.$ Composing along $\mathrm{ARings}_k \to \mathcal{S}$ gives an accessible limit preserving functor $C \to \mathcal{S},$ which must be corepresentable by an object $M \in C.$ We note that $M \simeq F(k[x]).$ This follows via adjunction and the fact that the forgetful functor $\mathrm{ARings}_k \to \mathcal{S}$ is corepresented by $k[x].$
    
    \item Finally, let $\mathrm{CAlg}(D(k))$ denote the $\infty$-category of commutative algebra object in the derived $\infty$-category of $k$-vector spaces, or equivalently $E_\infty$-algebras over $k.$ There is a colimit preserving functor $\mathrm{ARings}_k \to \mathrm{CAlg}(D(k))$ that we will somewhat abuseively call $\mathrm{id}.$
    
    \item We use cohomological conventions, i.e., the fullsubcategory of connective objects in $D(k)$ is denoted by $D(k)^{\le 0}.$ They are characterized by the property that the cohomology $H^i(\cdot)=0$ for $i>0.$
    
    \item For an object $T \in \mathrm{CAlg}(D(k)),$ we will let $\mathrm{LMod}_{T}$ denote the derived $\infty$-category of left modules over $T.$ Thinking of $T$ as a ring spectrum, they are modeled by $T$-module spectra. When $T$ is an ordinary ring, it is equivalent to the derived $\infty$-category $D(A),$ which can be modeled by chain complexes over $A.$

\end{enumerate}{}

\subsection{Acknowledgements} I am thankful to Benjamin Antieau, Bhargav Bhatt and Luc Illusie for helpful conversations and their encouraging comments on this paper. I also thank the support of NSF grant DMS \#1801689.

\section{Recovering the stacky approach from de Rham cohomology}
Let $A$ be a smooth algebra over $k.$ Let $\Omega_A^*$ denote the algebraic de Rham complex. Then $\Omega_A^*$ has the additional structure of a commutative differential graded algebra. The latter structure can be used to view $\Omega_A^*$ as an object in the more flexible $\infty$-category of commutative algebra objects in the derived $\infty$-category of $k$-vector spaces, denoted as $\mathrm{CAlg}(D(k)).$

This gives a functor $\mathrm{dR}: \mathrm{Alg}_k^{\mathrm{sm}} \to \mathrm{CAlg}(D(k))$ from the category of smooth $k$-algebras to $\mathrm{CAlg}(D(k)).$ Via left Kan extension along the inclusion $\mathrm{Alg}_k^{\mathrm{sm}} \to \mathrm{ARings}_k,$ we can equivalently view the above as a functor 

$$ \mathrm{dR}: \mathrm{ARings}_k \to \mathrm{CAlg}(D(k)). $$
Here and onwards, $\mathrm{ARings}_k$ denotes the $\infty$-category of animated $k$-algebras. \vspace{2mm}

So far, we have been forcefully employing the language of $\infty$-categories to the situation. But now it is time to collect the rewards. In the above set up, the functor $\mathrm{dR}$ preserves colimits. This follows from \cite[Prop.~5.5.8.15]{Lu} and was observed already in \cite{Bha12}. By the adjoint functor theorem, $\mathrm{dR}$ has a \textit{right adjoint}. Let us call the right adjoint 

$$\mathrm{dR}^{\vee}: \mathrm{CAlg}(D(k)) \to \mathrm{ARings}_k. $$

\vspace{2mm} 
Note that the category of ordinary $k$-algebras $\mathrm{Alg}_k$ is a full subcategory of $\mathrm{CAlg}(D(k)).$ By restricting the functor $\mathrm{dR}^{\vee}$ above, we get a functor

$$ \mathrm{dR}^{\vee}_{\circ}: \mathrm{Alg}_k \to \mathrm{ARings}_k. $$

\begin{theorem}\label{main} We have a natural isomorphism of ring stacks
$$\mathrm{dR}^{\vee}_{\circ} \simeq \mathrm{Cone}(\mathbb{G}_a^\sharp \to \mathbb{G}_a).$$
\end{theorem}{}

\begin{proof}
Note that $\mathrm{dR}$ is naturally equipped with the Hodge filtration, which induces an arrow $\mathrm{gr}^0: \mathrm{dR} \to \mathrm{id}$ in $\mathrm{Fun}^{\mathrm L} (\mathrm{ARings}_k, \mathrm{CAlg}(D(k))).$ Here, $\mathrm{id}$ really means the natural functor $\mathrm{id}: \mathrm{ARings}_k \to \mathrm{CAlg}(D(k)),$ which preserves colimits. By the adjoint functor theorem we get an arrow $\mathrm{id}^{\vee} \to \mathrm{dR}^{\vee}$ in $\mathrm{Fun}^{\mathrm R} (\mathrm{CAlg}(D(k)), \mathrm{ARings}_k).$ Restricting along $\mathrm{Alg}_k \to \mathrm{CAlg}(D(k))$ gives an arrow 
$$\mathrm{id}^{\vee}_{\circ} \to \mathrm{dR}^{\vee}_{\circ} $$in $\mathrm{Fun}(\mathrm{Alg}_k, \mathrm{ARings}_k).$

\begin{lemma}
We have $\mathrm{id}^{\vee}_{\circ} \simeq \mathbb{G}_a$ as ring stacks (\cref{not} (5)). In particular, both are representable by schemes.
\end{lemma}{}

\begin{proof}This is a definition chase that we omit. See \cref{not} (7).
\end{proof}{}
Therefore, we get a map $H: \mathbb{G}_a \to \mathrm{dR}^{\vee}_{\circ} $ of ring stacks. Note that taking the fibre (kernel) of the above map gives a functor $\mathrm{Ker}\, H: \mathrm{Alg}_k \to D(k)^{\le 0},$ where $D(k)^{\le 0}$ denotes the $\infty$-category of connective $k$-vector spaces (we use cohomological indexing) or in other words, animated $k$-vector spaces. We are going to identify $\mathrm{Ker}\, H$ explicitly.

\begin{lemma}[Bhatt]
We have an isomorphism $k \otimes_{\mathrm{dR}(k[x])} k[x] \simeq {D}_x(k[x]),$ where the latter denotes divided power envelope of $k[x]$ at the ideal $(x).$ The tensor product is taken along the map $\mathrm{gr}^0: \mathrm{dR}(k[x]) \to k[x].$
\end{lemma}{}

\begin{proof}
The key is to use that $$k \otimes_{\mathrm{dR}(k[x])} k[x] \simeq \mathbb L \mathrm{dR}_{k/k[x]},$$ where the right hand side denotes derived de Rham cohomology. The latter can be computed by using the conjugate filtration and the cotangent complex \cite[Lemma 3.29]{Bha12}.
\end{proof}{}

This shows that $\mathrm{Ker}\, H \simeq \mathbb{G}_a^\sharp$ as a functor from $\mathrm{Alg}_k \to D(k)^{\le 0}.$ In particular, $\mathrm{Ker}\, H$ is representable by an affine scheme.

\begin{lemma}\label{easy} The functor $\mathrm{dR}^{\vee}_{\circ}$ is an fpqc sheaf of animated rings.
\end{lemma}{}

\begin{proof} It is enough to check that the composite functor of $\mathrm{dR}^{\vee}_{\circ}$ along $\mathrm{ARings}_k \to \mathcal{S},$ which gives a functor $\mathrm{Alg}_k \to \mathcal{S}$ is a sheaf of spaces. But that functor sends $B \to \mathrm{Maps}_{\mathrm{CAlg}(D(k))}(\mathrm{dR}(k[x]), B),$ (see \cref{not}~(7)) and thus the claim follows by classical faithfully flat descent. 
\end{proof}{}

This constructs a map $\mathrm{Cone}(\mathbb{G}_a^\sharp \to \mathbb{G}_a) \to \mathrm{dR}^{\vee}_{\circ}$ in $\mathrm{Fun}(\mathrm{Alg}_k, \mathrm{ARings}_k).$ We needed the above lemma because formation of the cone on the LHS involves sheafification.\vspace{2mm}

Having constructed a map of ring stacks, now we need to check that they are isomorphic. This can be done at the level of stacks by forgetting the ring structure, i.e., we can check that $\mathrm{Cone}(\mathbb{G}_a^\sharp \to \mathbb{G}_a) \to \mathrm{dR}^{\vee}_{\circ}$ is an equivalence in $\mathrm{Fun}(\mathrm{Alg}_k, \mathcal{S}),$ after using the functor $\mathrm{ARings}_k \to \mathcal{S},$ where the latter denotes the $\infty$-category of spaces. This will rely on a Tannakian reconstruction result for $\mathcal{Y}:=\mathrm{Cone}(\mathbb{G}_a^\sharp \to \mathbb{G}_a).$ 

\begin{lemma}\label{tan} Let $B$ be a $k$-algebra. Then the groupoid of $\mathrm{Spec}\, B$ valued points of $\mathcal{Y},$ denoted as $\mathcal{Y}(B) \simeq \mathrm{Maps}_{\mathrm{CAlg}(D(k))}(R\Gamma(\mathcal{Y}, \mathscr O), B).$ 
\end{lemma}{}

We will give a proof of the above lemma in the next section. Granting this proof, we note that $R\Gamma(Y,\mathscr O) \simeq \mathrm{dR}(k[x]).$ To see this, we can compute the LHS by faithfully flat descent along $\mathbb{G}_a \to \mathcal{Y}$ and the RHS via the Cech-Alexander complex. This finishes the proof (\textit{c.f.} proof of \cref{easy}) of \cref{main} since we obtain $\mathcal{Y}(B) \simeq \mathrm{dR}^\vee_{\circ} (B).$
\end{proof}

\section{Tannakian reconstruction}
Throughout this section let $\mathcal{Y}:=\mathrm{Cone}(\mathbb{G}_a^\sharp \to \mathbb{G}_a).$ There is a natural map $\mathcal{Y} \to B\mathbb{G}_a^{\sharp}$ whose fiber is $\mathbb{G}_a.$ Since $\mathbb{G}_a$ is an affine scheme and $* \to B\mathbb{G}_a^{\sharp}$ is faithfully flat it follows that the map $\mathcal{Y} \to B\mathbb{G}_a^{\sharp}$ is an affine morphism of stacks.

\begin{lemma}\label{one}
The stack $B\mathbb{G}_a^\sharp$ has cohomological dimension $1.$
\end{lemma}{}
\begin{proof}
By a spectral sequence argument, it is enough to prove that if $F$ is a quasi-coherent sheaf on $B\mathbb{G}_a^{\sharp}$ then $H^i(B\mathbb G_a^\sharp, F)=0$ for $i>0.$ Since the dual ${\mathbb{G}_a^{\sharp}}^*$ of the group scheme $\mathbb{G}_a^{\sharp}$ is isomorphic to the formal group law $\widehat{\mathbb G}_a,$ such an $F$ corresponds to a nilpotent $k[[T]]$-module $V$ whose cohomology is computed by $\mathrm{Ext}^i _{k[[T]]}(k, V).$ We refer to \cite[Section 37.3.12]{HZW} for more details on such duality. The claim now follows from standard resolution of $k$ by free modules over $k[[T]]$ as that yields $\mathrm{Ext}^i _{k[[T]]}(k, V)=0$ for $i>1.$
\end{proof}{}

\begin{corollary}\label{dim}
The stack $\mathcal{Y}$ has cohomological diemnsion $1.$
\end{corollary}{}

\begin{proof}
Follows from the above lemma since $\mathcal{Y} \to B\mathbb{G}_a^\sharp$ is affine.
\end{proof}{}

\begin{lemma}
The structure sheaf $\mathscr O$ is a compact generator for the derived $\infty$-category of quasi-coherent sheaves on $B\mathbb{G}_a^\sharp,$ denoted as $D_{\mathrm{qc}}(B \mathbb{G}_a^\sharp).$
\end{lemma}{}

\begin{proof}
The structure sheaf $\mathscr O$ is compact since $B \mathbb{G}_a^\sharp$ has finite cohomological dimension. Proving that it is a generator amounts to showing that if $R \mathrm{Hom}(\mathscr O, F)=0$ for some $F \in D_{\mathrm{qc}}(B \mathbb{G}_a^\sharp),$ then $F = 0.$ In other words, if $R\Gamma (F)=0,$ then we need to show that $F=0.$ By the hypercohomology spectral sequence and \cref{one}, we get that $H^i (B \mathbb{G}_a^\sharp, \mathcal{H}^j F)=0$ for all $j$ and all $i\ge 0.$ In particular, $H^0 (B \mathbb{G}_a^\sharp, \mathcal{H}^j F)=0.$ But since $\mathbb{G}_a^\sharp$ is a unipotent group scheme, a non-trivial representation must have a fixed vector. This implies that $\mathcal{H}^j F=0$ for all $j$ and therefore $F =0,$ as desired.
\end{proof}{}

\begin{corollary}\label{com}
The structure sheaf $\mathscr O$ is a compact generator for $D_{\mathrm{qc}} (\mathcal{Y}).$
\end{corollary}{}

\begin{proof}
Follows in a way similar to the proof above along with the fact that $\mathcal{Y} \to B \mathbb G_a^\sharp$ is affine.
\end{proof}{}

\begin{corollary}\label{what}
There is an equivalence of symmetric monoidal stable $\infty$-categories $D_{\mathrm{qc}}(\mathcal{Y}) \simeq \mathrm{LMod}_{R\Gamma(\mathcal{Y}, \mathscr O)}.$
\end{corollary}{}

\begin{proof}The equivalence follows from \cref{com} (which is compatible with the symmetric monoidal structures since $(\cdot) \otimes (\cdot)$ preserves colimits in both variables) by using \cite[Thm 7.1.2.1]{Lur3}. \end{proof}{}

\begin{lemma}\label{connective}
The connective objects with respect to the standard $t$-structure on $D_{\mathrm{qc}}(\mathcal{Y})$ are generated under colimits by the structure sheaf $\mathscr O.$
\end{lemma}

\begin{proof}
Since $\mathscr O$ is connective, it follows that objects generated under colimits by $\mathscr O$ are all connective. Thus it would be enough to prove that if $F \in D_{\mathrm{qc}}(\mathcal{Y})$ is such that $H^i (B \mathbb{G}_a^\sharp, F)=0$ for $i <0,$ then $F$ must be coconnective, i.e., $\mathcal{H}^j (F)=0$ for $j<0.$ This again follows from the hypercohomology spectral sequence, \cref{dim} and the fact that a nonzero representation of $\mathbb G_a^\sharp$ must have a fixed vector, by unipotence.
\end{proof}{}

\begin{proof}[Proof of \cref{tan}]
By the Tannaka duality theorem \cite[Thm~ 9.2.0.2]{Lur3}, we know that $\mathrm{Maps}(\w{Spec}\, B, \mathcal{Y})$ can be identified with the full subcategory of $\mathrm{Fun}^{\otimes} (D_{\w{qc}}(\mathcal{Y}), D_{\mathrm{qc}}(\w{Spec}\, B) )$ spanned by those symmetric monoidal functors $F$ that has a right adjoint $G,$ such that $G$ preserves colimits; $F$ preserves connective objects, and $F$ and $G$ satisfies a projection formula, i.e., $u \otimes G(v) \simeq G(F(u)\otimes v).$ Note that \cref{connective} implies that $F$ preserving connective objects is implied by the necessary condition that $F (\mathscr{O}_{\mathcal{Y}}) \simeq \mathscr{O}_{\w{Spec}\, B}.$ Now we note that $D_\w{qc}(\w{Spec}\, B) \simeq \mathrm{LMod}_B$, and $D_\mathrm{qc}(\mathcal{Y}) \simeq \mathrm{LMod}_{R\Gamma(\mathcal{Y}, \mathscr O)}$ (by \cref{what}). Therefore the full subcategory of such functors in $\mathrm{Fun}^{\otimes} (D_{\w{qc}}(\mathcal{Y}), D_{\mathrm{qc}}(\w{Spec}\, B ))$ also corresponds to $\mathrm{Maps}_{\mathrm{CAlg}(D(k))}(R\Gamma(\mathcal{Y}, \mathscr O), B).$ This gives $\mathcal{Y}(B) \simeq \mathrm{Maps}_{\mathrm{CAlg}(D(k))}(R\Gamma(\mathcal{Y}, \mathscr O), B),$ as desired.
\end{proof}{}

\begin{remark}\label{easthall} We point out that the isomorphism $D_{\mathrm{qc}}(\mathcal{Y}) \simeq \mathrm{LMod}_{R\Gamma(\mathcal{Y}, \mathscr O)}$ in \cref{what} is very special to the stack $\mathcal{Y} = \mathbb{G}_a^{\w{dR}}$ that is being considered. For example, it is not true that $D_{qc}(B\mathbb{G}_a) \simeq \w{LMod}_{R\Gamma(B\mathbb{G}_a, \mathscr{O})}$ (over a perfect field $k$ of characteristic $p$). In fact, Hall, Neeman and Rydh proved \cite[Proposition 3.1]{HNR} that $D_{qc}(B \mathbb{G}_a)$ has no nonzero compact objects (negatively answering a question of Ben-Zvi). Their result immediately shows that such an isomorphism is impossible since $\w{LMod}_{R\Gamma(B\mathbb{G}_a, \mathscr{O})}$ is compactly generated. Therefore, the methods used in our proof above gives some techniques to prove certain positive results in this direction in specific situations.


\end{remark}{}

\newpage

\bibliographystyle{amsalpha}
\bibliography{main}

\end{document}